\documentclass[11pt]{amsart}

\usepackage{amscd}
\usepackage{amsmath}
\usepackage{amssymb}
\usepackage{amsthm}
\usepackage{epsf}
\usepackage{float}
\usepackage{latexsym}
\usepackage{verbatim}
\usepackage{microtype}
\usepackage{tikz}
\usepackage{hyperref}
\usetikzlibrary{positioning}
\usetikzlibrary{matrix}

\newtheorem{theorem}{Theorem}[section]

\newtheorem{definition}[theorem]{Definition}
\newtheorem{lemma}[theorem]{Lemma}

\newtheorem{corollary}[theorem]{Corollary}
\newtheorem{proposition}[theorem]{Proposition}
\newtheorem{question}[theorem]{Question}
\newtheorem{varexample}[theorem]{Example}
\theoremstyle{definition}
\newtheorem{remark}[theorem]{Remark}

\newcommand{\PP}{\mathbb{P} }

\newcommand{\RR}{\mathbb{R}}
\newcommand{\ZZ}{\mathbb{Z}}

\newcommand{\Pic}{\operatorname{Pic}}
\newcommand{\Trop}{\operatorname{Trop}}

\newcommand{\Jac}{\operatorname{Jac}}

\newcommand{\an}{\operatorname{an}}

\newcommand{\br}{\beta}

\newenvironment{example}{\begin{varexample}
\begin{normalfont}}{\end{normalfont}
\end{varexample}}
\begin{document}
\title{Lifting divisors on a generic chain of loops}
\address{Department of Mathematics \\ Yale University \\ PO Box 208283 \\ New Haven, CT 06520}
\email{dustin.cartwright@yale.edu}
\email{david.jensen@yale.edu}
\email{sam.payne@yale.edu}
\author{Dustin Cartwright}
\author{David Jensen}
\author{Sam Payne}
\date{}
\bibliographystyle{alpha}

\begin{abstract}
Let $C$ be a curve over a complete valued field with infinite residue field whose skeleton is a chain of loops with generic edge lengths.  We prove that
any divisor on the chain of loops that is rational over the value group lifts to a divisor of the same rank on $C$, confirming a conjecture of Cools,
Draisma, Robeva, and the third author.
\end{abstract}

\maketitle

\section{Introduction}

The tropical proof of the Brill-Noether Theorem in~\cite{tropicalBN} gives a classification of the special divisors on a chain of $g$ loops with
generic edge lengths.  The same argument works for a generic chain of loops with bridges, as studied in \cite{tropicalGP}, since it depends only on the
tropical Jacobian and the image of the Abel-Jacobi map.  Suppose $\Gamma$ is such a chain of loops with or without bridges, as shown in
Figure~\ref{TheGraph}, and $C$ is a curve over a complete valued field~$K$ with totally split reduction and skeleton~$\Gamma$.  The classification of special divisors on $\Gamma$ shows in particular that the polyhedral set $W^r_d(\Gamma) \subset \Pic_d(\Gamma)$ parametrizing divisor classes of rank at least $r$ has pure dimension equal to
\begin{equation}\label{Eq:Rho}
\rho(g,r,d) = g-(r+1)(g-d+r),
\end{equation}
which is the dimension of the analogous Brill-Noether locus $W^r_d(C)\subset \Pic_d(C)$ for the algebraic curve~$C$.

The real torus $\Pic_d(\Gamma)$ is canonically identified with the skeleton of the Berkovich analytic space $\Pic_d(C)^{\an}$ \cite{BakerRabinoff13}.  Furthermore, the tropicalization map given by retraction to this skeleton respects dimension \cite{Gubler07} and maps $W^r_d(C)^{\an}$ into $W^r_d(\Gamma)$
\cite{Baker08}, so the coincidence of dimensions suggests the possibility that $W^r_d(\Gamma)$ might be equal to the tropicalization of $W^r_d(C)$.  Our
main result confirms this possibility and proves \cite[Conj.~1.5]{tropicalBN} over complete fields with infinite residue field.

\begin{theorem}
\label{Thm:MainThm}
Let $K$ be a complete field with infinite residue field, and let $C$ be a smooth projective curve of genus~$g$ over $K$.  If $C$ has totally split reduction and skeleton is isometric to~$\Gamma$ then every divisor class on $\Gamma$ that is rational over the value group of $K$ lifts to a divisor class of the same rank on $C$.
\end{theorem}

\begin{figure}[h!] \label{Fig:ChainOfLoops}
\begin{tikzpicture}

\draw [ball color=black] (-1.7,-0.45) circle (0.55mm);
\draw (-1.95,-0.65) node {\footnotesize $v_1$};
\draw (-1.5,0) circle (0.5);
\draw (-1,0)--(0,0.5);
\draw [ball color=black] (-1,0) circle (0.55mm);
\draw (-0.85,0.3) node {\footnotesize $w_1$};
\draw (0.7,0.5) circle (0.7);
\draw (1.4,0.5)--(2,0.3);
\draw [ball color=black] (1.4,0.5) circle (0.55mm);
\draw [ball color=black] (0,0.5) circle (0.55mm);
\draw (-0.2,0.75) node {\footnotesize $v_2$};
\draw (2.6,0.3) circle (0.6);
\draw (3.2,0.3)--(3.87,0.6);
\draw [ball color=black] (2,0.3) circle (0.55mm);
\draw [ball color=black] (3.2,0.3) circle (0.55mm);
\draw [ball color=black] (3.87,0.6) circle (0.55mm);
\draw (4.5,0.3) circle (0.7);
\draw (5.16,0.5)--(5.9,0);
\draw (6.4,0) circle (0.5);
\draw [ball color=black] (5.16,0.5) circle (0.55mm);
\draw (5.48,0.74) node {\footnotesize $w_{g-1}$};
\draw [ball color=black] (5.9,0) circle (0.55mm);
\draw [ball color=black] (6.9,0) circle (0.55mm);
\draw (5.7,-.2) node {\footnotesize $v_g$};
\draw (7.3,-.2) node {\footnotesize $w_g$};

\draw [<->] (3.3,0.4) arc[radius = 0.715, start angle=10, end angle=170];
\draw [<->] (3.3,0.2) arc[radius = 0.715, start angle=-9, end angle=-173];

\draw (2.5,1.25) node {\footnotesize$\ell_i$};
\draw (2.75,-0.7) node {\footnotesize$m_i$};
\end{tikzpicture}
\caption{The bridges of $\Gamma$ have arbitrary non-negative lengths and the edges of the cycles have generic positive lengths.  The
precise genericity condition is given at the beginning of Section~\ref{Sec:ChainOfLoops}.}
\label{TheGraph}
\end{figure}

\noindent One interesting special case is when $\rho(g,r,d)$ is zero.  In this case, the chain of loops $\Gamma$ has exactly
\begin{equation}\label{Eq:Cardinality}
\lambda(g,r,d) = g! \prod_{i=0}^r \frac{i!}{(g-d+r+i)!}
\end{equation}
divisor classes of degree $d$ and rank $r$, which is equal to the number of divisor classes of degree $d$ and rank $r$ on a general curve of genus $g$
\cite{GriffithsHarris80}.  Theorem~\ref{Thm:MainThm} then shows that the tropicalization map from $W^r_d(C)$ to $W^r_d(\Gamma)$ is bijective.

\begin{remark}
The special case of Theorem~\ref{Thm:MainThm} where $r = 1$ and $\rho(g,r,d) = 0$ is due to Cools and Coppens, who gave a proof using the basepoint free pencil trick \cite{CoolsCoppens14}.  Kawaguchi and Yamaki recently proved analogous lifting theorems for special divisors on skeletons of totally
degenerate hyperelliptic curves and non-hyperelliptic curves of genus 3.  Their results include the special cases of Theorem~\ref{Thm:MainThm} where $g$ is
at most $3$ \cite{KawaguchiYamaki13, KawaguchiYamaki14}.
\end{remark}

While Theorem~\ref{Thm:MainThm} establishes Conjecture~1.5 from~\cite{tropicalBN} for complete valued fields with infinite residue field, the conjecture can be false when the residue field is finite, as in the following example.

\begin{example}
Suppose $\Gamma$ is a chain of $g \geq 2$ loops with generic integral edge
lengths and no bridges, as in \cite{tropicalBN}. Let $K = \mathbb F_3((t))$ with
the $t$-adic valuation.  By deformation theory, there exists a smooth projective
curve $C$ over $K$ with totally split reduction that has skeleton isometric
to~$\Gamma$~\cite[Thm.~B.2]{Baker08}.  Let $[D]$ be the class of the vertex~$v_2$ as in Figure~\ref{TheGraph}, which is the unique
effective divisor in its linear equivalence class. If $\mathfrak C$ denotes the
regular semistable model of~$C$, then a lift of~$[D]$ to $C$ would give rise to an effective divisor whose closure in~$\mathfrak C$ intersects
the component corresponding to~$v_2$ with multiplicity~$1$, and hence the intersection would be a smooth point of this component defined over $\mathbb
F_3$.  However, this component is isomorphic to $\PP^1_{\mathbb F_3}$ and $v_2$ has degree $4$ in~$\Gamma$, so all of its $\mathbb F_3$-rational points are
nodes.  We conclude that there is no such lift.
\end{example}

Theorem~\ref{Thm:MainThm} and \cite[Thm.~1.3]{tropicalBN} for the chain of loops are together somewhat analogous to the regeneration theorem of Eisenbud and Harris
\cite[Thm.~3.4]{EisenbudHarris86} for nodal curves of compact type. The regeneration theorem says that the space of limit linear series on a nodal curve of compact type has dimension everywhere locally at least the
expected dimension, and a point where the local dimension is exactly the expected dimension is the limit of a linear series on the general fiber of a
one-parameter smoothing.  The following questions ask in what generality such analogues of the regeneration theorem should hold for arbitrary metric
graphs.

\begin{question}
\label{Conj:LocalDim}
If $\Gamma'$ is a metric graph of genus $g$, is the dimension of $W^r_d(\Gamma')$ everywhere locally at least $\rho(g,r,d)$?
\end{question}

\begin{question}
\label{Conj:TropicalRegeneration}
Let $C$ be a totally degenerate curve over an algebraically closed field $K$ with skeleton $\Gamma'$. Suppose $W^r_d(\Gamma')$ has local dimension $\rho(g,r,d)$ at the class of a
divisor $D$ and that $D$ is rational over the value group of $K$. Is there necessarily a divisor of degree $d$ and rank $r$ on
$C$ that tropicalizes to~$D$?
\end{question}

\noindent The answers to Questions~\ref{Conj:LocalDim} and Question~\ref{Conj:TropicalRegeneration} are affirmative for chains of loops with generic edge lengths, as shown in \cite{tropicalBN} and the present paper, respectively.  Both questions are wide open for arbitrary graphs, and no counterexamples are known.

We now explain the general strategy used in the proof of Theorem~\ref{Thm:MainThm}.  Because of the totally split reduction and infinite residue field, the curve $K$ has infinitely many $K$-points.  We can therefore choose a base point on the curve so that each
component of the Picard scheme of $C$ is identified with the Jacobian.  We then study the Brill-Noether loci as subschemes
\[
W^r_d(C) \subset \Jac(C).
\]
Since $C$ is totally degenerate, the universal cover of $\Jac(C)^{\an}$ gives a uniformization
\[
T^{\an} \rightarrow \Jac(C)^{\an},
\]
by an algebraic torus $T$ of dimension $g$.  The tropicalization of this torus is the universal cover of the skeleton of $\Jac(C)$, which is canonically
identified with the tropical Jacobian of $\Gamma$ \cite{BakerRabinoff13}.

One of our key tools is Rabinoff's lifting theorem~\cite{Rabinoff12}, which we apply to the analytic preimages in~$T$ of algebraic subschemes of $\Jac(C)$.
This lifting theorem says that isolated points in complete intersections of tropicalizations of analytic hypersurfaces lift to
points in the analytic intersection with appropriate multiplicities.  We apply it to translates of the preimage of the theta divisor $\Theta_{\Gamma} =
W^0_{g-1} (\Gamma)$, as follows.

The Baker-Norine definition of rank from \cite{BakerNorine07} implicitly expresses $W^r_d(\Gamma)$ as an intersection of translates of $\Theta_{\Gamma}$.
In Proposition~\ref{Prop:Containments}, we reinterpret this construction scheme-theoretically and show that the local equations for the corresponding
translates of $\Theta_C$ vanish on $W^r_d(C)$.  When $\rho$ is zero, we produce explicit translates of $\Theta_C$ whose tropicalizations intersect
transversally with multiplicity 1 at a given point of $W^r_d(\Gamma)$.  By Rabinoff's lifting theorem, applied on the universal cover of $\Jac(C)^{\an}$,
there is exactly one point in this complete intersection over the point of $W^r_d(\Gamma)$.  The rest of this complete intersection is typically
larger than $W^r_d(\Gamma)$, but the argument shows that the tropicalization map $W^r_d(C) \rightarrow W^r_d(\Gamma)$ is injective.  Since the two sets
have the same cardinality, we conclude that it is bijective.

When $\rho$ is positive, we consider a point $x$ in a dense subset of $W^r_d(\Gamma)$ and choose $\rho$ additional translates of $\Theta_{\Gamma}$
that meet $W^r_d(\Gamma)$ transversally in the expected number of points. For each point in this intersection, we combine the $\rho$ previously chosen
translates of $\Theta_{\Gamma}$ with $g - \rho$ additional translates that contain $W^r_d(C)$, and then apply a similar lifting and counting argument to
conclude that $x$ is in the image of $W^r_d(C)^{\an}$.  Since the tropicalization map $W^r_d(C)^{\an} \rightarrow W^r_d(\Gamma)$ is proper and its
image contains a dense subset, we conclude that it is surjective.  We then study the initial degenerations of~$W^r_d(C)$, and show that these are rational over the residue field of~$K$. Finally, since the residue field is infinite, we conclude that these initial degenerations have smooth rational points, which lift to points over~$K$ by Hensel's Lemma.

\medskip

The main result in~\cite{tropicalGP} shows that if the bridges in~$\Gamma$ have positive length, then $C$ is Gieseker-Petri general, and hence
$W^r_d(C)$ is smooth away from $W^{r+1}_d(C)$.  Our argument also gives the weaker statement that $W^r_d(C)$ is generically smooth, but without the hypothesis that the graph~$\Gamma$ has bridges. More precisely, we have:

\begin{proposition}\label{Prop:Reduced}
If $C$ is a curve over a valued field~$K$ with skeleton isometric to~$\Gamma$, then $W^r_d(C)$ is reduced.
In particular, if $\rho(g,r,d)$ is zero then $W^r_d(C)$ is smooth.
\end{proposition}

\begin{remark}
Our approach in Theorem~\ref{Thm:MainThm} is inspired by the tropical scheme theory of Giansiracusa and Giansiracusa \cite{GiansiracusaGiansiracusa13} and the tropical Hilbert-Chow
morphism of Maclagan and Rinc\'{o}n \cite{MaclaganRincon14}.  Although the results of those papers are not used in the proofs, our main arguments relating
various multiplicities to expressions of tropical Brill-Noether loci as intersections of translates of the tropical theta divisor grew out of a desire to
understand $W^r_d(\Gamma)$ as a tropical scheme.
\end{remark}

\begin{remark}
When finishing this paper, we learned of an independent proof of Theorem~\ref{Thm:MainThm} by Amini and Baker.  Their approach uses specialization through a point of $\mathcal{M}_g$ over a rank 2 valuation ring, in which the general fiber is smooth, the intermediate fiber is a chain of elliptic curves, and the special fiber is totally degenerate. They then proceed by connecting the tropical theory on the special fiber with the classical theory of limit linear series on the intermediate fiber. The methods are disjoint and complementary, and we expect that both approaches will be fruitful for future applications, perhaps in combination.
\end{remark}

\noindent \textbf{Acknowledgments.}
We would like to thank O.~Amini, M.~Baker, N.~Giansiracusa, and W.~Gubler for
helpful conversations.  The second author was supported in part by an AMS-Simons
Travel Grant.  The third author was supported in part by NSF DMS--1068689 and by NSF CAREER DMS--1149054.

\section{Brill-Noether loci on the chain of loops}\label{Sec:ChainOfLoops}

We begin this section by reviewing the classification of special divisors on the chain of loops with generic edge lengths $\Gamma$, following
\cite{tropicalBN} to which we refer the reader for proofs and further details. As in that paper, we always assume that none of the ratios $\ell_i/m_i$ is
equal to the ratio of two positive integers whose sum is at most $2g-2$. We then build up to Proposition~\ref{Prop:TropicalLocalEqns}, which describes the
local structure of $W^r_d(\Gamma)$ near the class of a vertex avoiding divisor as a complete intersection of translates of the tropical theta divisor, and
Proposition~\ref{Prop:IntersectionNumber}, which generalizes the counting formula (\ref{Eq:Cardinality}) to the case where the dimension of $W^r_d(\Gamma)$
is positive.

\subsection{Classification of special divisors on \texorpdfstring{$\Gamma$}{Gamma}} \label{sec:classification}

Recall that an effective divisor~$D$ on a metric graph $\Gamma'$ is $v$-reduced, where $v$ is a point in $\Gamma'$, if the multiset of distances from $v$
to points in~$D$ is lexicographically minimal among all
effective divisors equivalent to~$D$. The $v_1$-reduced divisors
on~$\Gamma$ are classified as follows. For each~$i$, let $\gamma_i$ be
the $i$th loop of~$\Gamma$ minus $v_i$, and let $\br_i$ be the half-open bridge $(w_i, v_{i+1}]$.  Then
$\Gamma$ decomposes as a disjoint union
\[
\Gamma = \{ v_1 \} \sqcup \gamma_1 \sqcup \br_1 \sqcup \cdots \sqcup \gamma_g,
\]
as shown in Figure~\ref{fig:decomposition}, and an effective divisor is
$v_1$-reduced if and only if it contains at most one point on each of the
punctured loops~$\gamma_i$ and no points on the bridges $\br_i$.


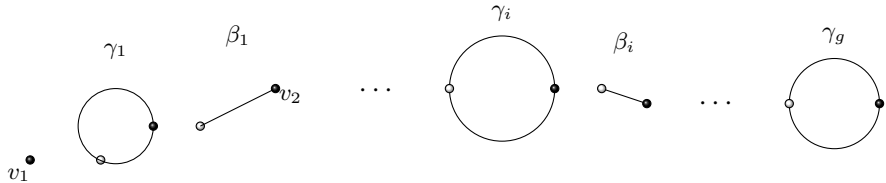
\begin{figure}[H]
\begin{tikzpicture}
\matrix[column sep=0.5cm] {
\begin{scope}[grow=right, baseline]
\draw [ball color = black] (-2,-0.45) circle (0.55mm);
\draw node at (-2.15,-0.65) {\footnotesize $v_1$};
\end{scope}
&
\begin{scope}[baseline]
\draw [ball color=white] (-1.7,-0.45) circle (0.55mm);
\draw (-1.5,0) circle (0.5);
\draw [ball color=black] (-1,0) circle (0.55mm);
\draw (-1.5,1.0) node {\footnotesize $\gamma_1$};
\end{scope}
&
\begin{scope}[grow=right,baseline]
\draw [ball color=white] (-1,0) circle (0.55mm);
\draw (-1,0)--(0,0.5);
\draw [ball color=black] (0,0.5) circle (0.55mm);
\draw (0.2,0.4) node {\footnotesize $v_2$};
\draw (-0.5,1.2) node {\footnotesize $\br_1$};
\end{scope}
&
\begin{scope}[grow=right,baseline]
\draw node at (.5,0.48) {$\cdots$};
\end{scope}
&
\begin{scope}[grow=right,baseline]
\draw (0.7,0.5) circle (0.7);
\draw [ball color=black] (1.4,0.5) circle (0.55mm);
\draw [ball color=white] (0,0.5) circle (0.55mm);
\draw (0.7,1.5) node {\footnotesize $\gamma_i$};
\end{scope}
&
\begin{scope}[grow=right,baseline]
\draw (1.4,0.5)--(2,0.3);
\draw [ball color=black] (2,0.3) circle (0.55mm);
\draw [ball color=white] (1.4,0.5) circle (0.55mm);
\draw (1.7,1.1) node {\footnotesize $\br_i$};
\end{scope}
&
\begin{scope}[grow=right,baseline]
\draw node at (2,0.3) {$\cdots$};
\end{scope}
&
\begin{scope}[grow=right,baseline]
\draw (2.6,0.3) circle (0.6);
\draw [ball color=white] (2,0.3) circle (0.55mm);
\draw [ball color=black] (3.2,0.3) circle (0.55mm);
\draw (2.6,1.2) node {\footnotesize $\gamma_g$};
\end{scope}
\\};
\end{tikzpicture}
\caption{A decomposition of $\Gamma$.}\label{fig:decomposition}
\end{figure}

Since every effective divisor on~$\Gamma$ is equivalent to a unique
$v_1$-reduced divisor, each effective divisor class is represented uniquely by a vector $(d_0, x_1,
\ldots, x_g)$, where $d_0$ is the coefficient of~$v_1$, and $x_i
\in \RR/(\ell_i + m_i) \mathbb Z$ is the distance from $v_i$ to the chip on
the $i$th punctured loop~$\gamma_i$, measuring counterclockwise, if such a chip
exists and $x_i$ is set to $0$ otherwise. An effective divisor together with a positive integer $r$ also determines a lingering lattice path, which is a
sequence $p_0, \ldots, p_g$ of points in $\ZZ^r$, as follows.

\begin{definition}
Let $D$ be the $v_1$-reduced divisor represented by the data $(d_0 , x_1 , \ldots , x_g )$.  Then the associated \textit{lingering lattice path} $P$ in
$\ZZ^r$ starts at $p_0 = (d_0 , d_0 -1, \ldots , d_0 -r+1)$ with the $i$th step given by
\begin{displaymath}
p_i - p_{i-1} = \left\{ \begin{array}{ll}
(-1,-1, \ldots , -1) & \textrm{if $x_i = 0$}\\
e_j & \textrm{if $x_i = (p_{i-1}(j) +1)m_i$ mod $\ell_i + m_i$} \\
  & \textrm{and both $p_{i-1}$ and $p_{i-1} + e_j$ are in $\mathcal{C}$}\\
0 & \textrm{otherwise}
\end{array} \right\}
\end{displaymath}
where $e_0 , \ldots e_{r-1}$ are the standard basis vectors in $\mathbb{Z}^r$ and $\mathcal{C}$ is the open Weyl chamber
$$ \mathcal{C} = \{ y \in \mathbb{Z}^r \mid y_0 > y_1 > \cdots > y_{r-1} > 0 \}. $$
\end{definition}

\begin{proposition}
\label{Prop:LLP}
\cite[Theorem 4.6]{tropicalBN}
A divisor $D$ on $\Gamma$ has rank at least $r$ if and only if the associated lingering lattice path lies entirely in the open Weyl chamber $\mathcal{C}$.
\end{proposition}

\noindent The steps where $p_i-p_{i-1} = 0$ are called \textit{lingering steps},
and the number of lingering steps is at most the Brill-Noether number
$\rho(g,r,d)$ from~(\ref{Eq:Rho}).

Our representation of divisors is compatible with
the Abel-Jacobi map, as follows.
We define an orientation on $\Gamma$ by
orienting each of the loops $\gamma_i$ counter-clockwise, and define a basis of
1-forms on $\Gamma$ by setting $\omega_i = d\gamma_i$.  We then have the
Abel-Jacobi map~\cite[Sec. 6]{MikhalkinZharkov08}:
\begin{equation*}
AJ\colon \Pic_d ( \Gamma ) \to \Jac ( \Gamma ) = \prod_{i=1}^g \mathbb{R} /(m_i + \ell_i ) \mathbb{Z},
\end{equation*}
given by
\begin{equation*}
AJ\bigg(\sum_{j=1}^d p_j \bigg) := \sum_{i=1}^g \bigg( \sum_{j=1}^d
\int_{v_1}^{p_j} \omega_i \bigg) e_i ,
\end{equation*}
where $e_i$ denotes the $i$th standard basis vector in $\mathbb{R}^g$.  Specifically, the Abel-Jacobi map sends the divisor corresponding to the data $(d_0
, x_1, \ldots, x_g)$ to the point
\begin{equation*}
\sum_{i=1}^g (n_i m_i + x_i)e_i \in \prod_{i=1}^g \mathbb{R} /(m_i + \ell_i ) \mathbb{Z},
\end{equation*}
where
$n_i = \# \{ j \in \mathbb{Z} \mid i<j\leq g, x_j \neq 0 \} . $  Together with Proposition~\ref{Prop:LLP}, this
tells us that $W^r_d ( \Gamma )$ is a union of translates of the images of the
coordinate $\rho$-planes in $\mathbb{R}^g$, one for each lingering lattice path with $\rho$ lingering steps.  Given such a path, if the $i$th step is not lingering then the $i$th coordinate is fixed at $(p_{i-1} (j)+1+n_i)m_i$, while the remaining $\rho$ coordinates corresponding to lingering steps are allowed to move freely.
This is illustrated for
$(g,r,d)=(3,0,2)$ in Figure~\ref{fig:cube}.

\begin{figure}
\includegraphics{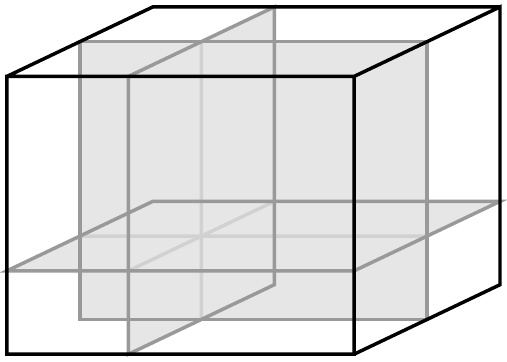}
\caption{The theta divisor in $\Pic^2(\Gamma)$, where $\Gamma$ is the chain of
$3$ loops. This theta divisor is the set of all effective divisors classes of
degree~$2$. Here $\Pic^2(\Gamma) \equiv (\RR/\ZZ)^3$ is obtained by identifying opposite faces of the pictured
cube.}\label{fig:cube}
\end{figure}

\subsection{Vertex avoiding divisors}

The description of $W^r_d(\Gamma)$ above identifies this tropical Brill-Noether locus with a union of subtori of dimension $\rho$, corresponding to lingering lattice paths with $\rho$ lingering steps.  Our analysis of lifting divisor classes is simplest away from the pairwise intersections of
these tori. The divisors in the complement of these intersections are called vertex avoiding, and have the following characterization.

\begin{definition}
A divisor class $[D]$ on $\Gamma$ of rank $r$ is \emph{vertex avoiding} if the following conditions hold:
\begin{enumerate}
\item  the associated lingering lattice path has exactly $\rho$ lingering steps,
\item  $x_i \neq m_i$ mod $\ell_i + m_i$ for any $i$, and
\item  $x_i \neq (p_{i-1}(j))m_i$ mod $\ell_i + m_i$ for any $i$.
\end{enumerate}
\end{definition}

The set of vertex avoiding divisor classes of degree $d$ and rank $r$ form a dense open subset of $W^r_d ( \Gamma )$.  In particular, if
$\rho(g,r,d)$ is zero then every divisor class in $W^r_d(\Gamma)$ is vertex avoiding.  Such divisors have the following nice property.

\begin{proposition}
\label{Prop:VertexAvoiding}
Let $[D]$ be a vertex avoiding divisor class of rank $r$ on $\Gamma$.  Then for each $0 \leq j \leq r$, there exists a unique divisor $D_j \sim D$ such that $D_j - jv_1 - (r-j)w_g$ is
effective. Moreover,
\begin{enumerate}
\item  $D_j - jv_1 - (r-j)w_g$ has no points on any of the bridges $\br_i$ or vertices $w_i$.
\item  For $j<r$, $D_j$ fails to have a point on $\gamma_i$ if and only if the $i$th step of the associated lingering lattice path is in the direction
    $e_j$.
\item  $D_r$ fails to have a point on $\gamma_i$ if and only if the $i$th step of the associated lingering lattice path is in the direction $(-1, \ldots
    , -1)$.
\end{enumerate}
\end{proposition}

\begin{proof}
The argument is identical to the proof of \cite[Prop.~6.2]{tropicalGP}, which is the special case where $\rho(g,r,d)$ is zero.
\end{proof}

\subsection{The local structure of \texorpdfstring{$W^r_d (\Gamma)$}{Wrd(Gamma)}}

We now show that, in a neighborhood of any
vertex avoiding divisor class, we can describe $W^r_d(\Gamma)$ as an
intersection of $g - \rho$ translates of the theta
divisor $\Theta_{\Gamma} = W^0_{g-1}(\Gamma)$. More specifically, we construct $g-\rho$ divisors
of the form $E_i - E_i'$, where $E_i$ is effective of degree~$r$ and $E'_i$ is
effective of degree $g-d+r-1$, such that a divisor class $[D']$ near $[D]$
has rank $r$ if and only if $[D'-(E_i-E'_i)]$ is effective for all~$i$.

We first describe the open neighborhoods which we will use in our proof. For $i
= 1, \ldots, g$, fix $p_i$ in the interior of one of the two edges connecting
$v_i$ to $w_i$. Then, if $\epsilon$ is sufficiently small, the
$\epsilon$-neighborhood of~$p_i$ consists of an interval and we define
\begin{equation*}
B(0, \epsilon) := \bigg\{ \sum_{i=1}^g q_i - \sum_{i=1}^g p_i \in \Jac ( \Gamma ) \mid d(p_i,
q_i) < \epsilon \bigg\} .
\end{equation*}

\begin{proposition}
The map sending $(q_1,\ldots, q_g)$ to $\sum_{i=1}^g q_i - \sum_{i=1}^g p_i \in
B(0, \epsilon)$
defines a homeomorphism between a product of open intervals and an open
neighborhood of zero in $\Jac (\Gamma)$.
\end{proposition}

\begin{proof}
It suffices to show that the Abel-Jacobi map sending the divisor $\sum_{i=1}^g q_i$ to its class is injective.  To see this, suppose that $\sum_{i=1}^g q_i
\sim \sum_{i=1}^g q'_i$.  The divisor $\sum_{i=1}^g q_i$ is $q'_i$-reduced, and hence has positive $q'_i$-degree if and only if $q_i = q'_i$.  It follows
that $q_i = q'_i$ for all $i$.
\end{proof}

\begin{definition}
For any divisor~$D$ of degree~$d$ on~$\Gamma$, we define
$$B(D, \epsilon) = D+B(0, \epsilon),$$ which is an open neighborhood of $D \in \Pic_d(\Gamma)$.
\end{definition}

We now return to our description of the Brill-Noether locus.

\begin{proposition}
\label{Prop:TropicalLocalEqns}
Let $[D] \in W^r_d ( \Gamma )$ be a vertex avoiding divisor class.  Then there are effective divisors $E_1 , \ldots , E_{g-\rho}$ of degree~$r$ and effective divisors $E'_1 , \ldots , E'_{g-\rho}$ of degree $g-d+r-1$ such
that for some $\epsilon > 0$,
$$ W^r_d ( \Gamma ) \cap B(D,\epsilon) = \bigcap_{i=1}^{g-\rho} [\Theta_{\Gamma} + E_i - E'_i ] \cap B(D,\epsilon) . $$
\end{proposition}

\begin{proof}
For $0 \leq i \leq r - 1$, let $A_i \subset \{ 1, \ldots , g \}$ denote the set of steps of the associated lingering lattice path in the direction~$e_i$, and let $A_r$ be the set of steps
in the direction $(-1, \ldots , -1)$.  For $0 \leq i \leq r$ and $j \in A_i$, let $E_i = iv_1 + (r-i)w_g$ and $E'_{i,j} = \sum_{k \in A_i , k \neq j} p_k$.
We show that
$$ W^r_d ( \Gamma ) \cap B(D,\epsilon) = \bigcap_{i=0}^r \bigcap_{j \in A_i} [ \Theta_{\Gamma} +E_i - E'_{i,j} ] \cap B(D,\epsilon) . $$
It follows from the definition of the rank of a divisor that the left-hand-side
is contained in the right-hand side. Indeed,
if some divisor $D'$ has rank at least $r$, then for any~$i$, $D'-E_i$ must be
linearly equivalent to an effective divisor, and thus the same is true for
$D'-E_i+E'_{i,j}$.

It remains to show that the right-hand side is contained in the left-hand side.  Note that, by definition, any divisor class $[D'] \in B(D,\epsilon)$ is of
the form
$$ [D'] = [D - \sum_{k=1}^g p_k + \sum_{k=1}^g q_k] $$
for some $q_k$ within $\epsilon$ of~$p_i$.  We then have
$$ D'-E_i+E'_{i,j} \sim D_i -E_i + \sum_{k=1}^g q_k - \sum_{k \notin A_i} p_k - p_j . $$
We write $D'_{i,j}$ for the divisor on the right.  By Proposition~\ref{Prop:VertexAvoiding}, we know that $D_i -E_i$ fails to have a point on $\gamma_k$ if
and only if $k \in A_i$.  It follows that the restriction $D'_{i,j} \vert_{\gamma_k}$ has degree 1 if and only if $k \neq j$, and $D'_{i,j}
\vert_{\gamma_j}$ has degree 0.

More precisely, for all $k \notin A_i$, $D'_{i,j} \vert_{\gamma_k}$ is
equivalent to a point within $\epsilon$ of the point of $D_i$ on $\gamma_k$.  Similarly, for $k \in A_i$, $k \neq j$, $D'_{i,j} \vert_{\gamma_k} = q_k$.
We therefore see that, for $k \neq j$ and $\epsilon$ sufficiently small, $D'_{i,j} \vert_{\gamma_k}$ is not equivalent to either of the vertices $v_k$ or
$w_k$.  It follows that $[D'_{i,j}]$ is effective if and only if $D'_{i,j} \vert_{\gamma_j}$ is effective.
We must therefore have $q_j = p_j$.  By varying over all $i$ and $j$, we see that $q_j = p_j$ for all non-lingering steps $j$, and therefore $D' \in W^r_d
( \Gamma )$ by Proposition~\ref{Prop:LLP}.
\end{proof}

\subsection{A counting formula when \texorpdfstring{$\rho$}{rho} is positive}

When $\rho(g,r,d) = 0$, the cardinality of $W^r_d(\Gamma)$ agrees with the cardinality of $W^r_d (C)$, and both are given by
formula~(\ref{Eq:Cardinality}). We conclude this section with an analogous counting computation when $\rho(g,r,d)$ is positive.

If $C$ is any algebraic curve of genus~$g$ such that $W^r_d(C)$ is
$\rho$-dimensional, where $\rho = \rho(g,r,d)$, then by \cite[Thm.~V.1.3]{ACGH},
the class of $W^r_d (C)$ is
\begin{equation*}
w^r_d = \left(\prod_{i=0}^r \frac{i!}{(g-d+r+i)!}\right) \cdot \Theta^{g-\rho}.
\end{equation*}
It follows that, if $W^r_d (C)$ is generically reduced and $\rho$-dimensional, as is the case when $C$ is general by \cite{GriffithsHarris80}, then its
intersection with $\rho$ general translates of $\Theta$ will consist of
$$ w^r_d \cdot \Theta^{\rho} = g! \prod_{i=0}^r \frac{i!}{(g-d+r+i)!} $$
distinct points.  We now prove the analogous counting formula for intersections of $W^r_d ( \Gamma )$ with translates of the tropical theta divisor.

\begin{proposition}
\label{Prop:IntersectionNumber}
The intersection of $W^r_d ( \Gamma )$ with $\rho$ general translates of $\Theta_{\Gamma}$ consists of
$$ g! \prod_{i=0}^r \frac{i!}{(g-d+r+i)!} $$
distinct points.  Moreover, if $[D] \in W^r_d (\Gamma)$ is vertex avoiding, then the $\rho$ translates can be chosen to all contain $[D]$.
\end{proposition}

\begin{proof}
When $\rho = 0$, any lingering lattice path has no lingering steps, and the number of rank $r$ lattice paths with precisely $(r+1)(s+1)$ steps and no
lingering steps is
$$ \Psi (r,s) := [(r+1)(s+1)]! \prod_{i=0}^r \frac{i!}{(s+1+i)!} . $$
We can use this to count the number of lingering lattice paths with precisely $\rho$ lingering steps.  To construct such a path, one can first choose the
$\rho$ lingering steps, and then choose a non-lingering lattice path on the remaining $g-\rho$ loops.  It follows that the number of such paths is
$$ {{g}\choose{\rho}} \Psi (r,g-d+r-1) = {{g}\choose{\rho}} (g-\rho)! \prod_{i=0}^r \frac{i!}{(g-d+r+i)!} . $$
As explained above, $W^r_d ( \Gamma )$ is a union of precisely this number of translates of $\rho$-dimensional coordinate tori.  Recall that the theta
divisor $\Theta_{\Gamma}$ is a union of translates of the $g$ coordinate $(g-1)$-dimensional tori.  A given $\rho$-dimensional coordinate torus will
therefore intersect $\rho$ general translates of $\Theta_{\Gamma}$ in $\rho !$ points.  It follows that the intersection of $W^r_d ( \Gamma )$ with such
translates is simply a union of
$$ {{g}\choose{\rho}} \rho ! \Psi (r,g-d+r-1) = g! \prod_{i=0}^r \frac{i!}{(g-d+r+i)!} $$
distinct points.

For the last statement in the theorem, fix a vertex avoiding divisor class $[D]$ in $W^r_d(\Gamma)$, and let $a_1 , \ldots , a_{\rho}$ be the lingering
steps of the lattice path associated to $D$.  For each $1 \leq i \leq \rho$, let $E_i$ be a sum of $g-1$ distinct points, one from each loop $\gamma_j$
other than $\gamma_{a_i}$.  The intersection $W^r_d (\Gamma) \cap [\Theta_{\Gamma} + (D-E_i )]$ contains $[D]$, and if the $E_i$ are chosen sufficiently
general then this intersection is transverse.
\end{proof}

\section{Lifting divisors}

In this section, we use our results on the structure of the tropical
Brill-Noether loci to prove Theorem~\ref{Thm:MainThm}. We begin by proving that the tropical theta divisor is multiplicity-free.

\subsection{Tropical multiplicities on Brill-Noether loci}

Section~\ref{sec:classification} gives an explicit description of the Brill-Noether
locus $W^0_d(\Gamma)$ with $r = 0$.  Since the tropical Jacobian of $\Gamma$ is canonically identified with the skeleton of $\Jac(C)$, and this
identification is compatible with Abel-Jacobi maps \cite{BakerRabinoff13}, we see that $W^0_d(\Gamma)$ is the tropicalization of $W^0_d(C)$.  We now
compute the tropical multiplicities on the facets of the tropical theta divisor $\Theta_{\Gamma} = W^0_{g-1}(\Gamma)$.

\begin{lemma}
\label{Lemma:MultiplicityOfTheta}
Every facet of the tropical theta divisor $\Theta_{\Gamma}$ has multiplicity 1.
\end{lemma}

\begin{proof}
As a consequence of Proposition~\ref{Prop:LLP}, $\Theta_{\Gamma}$ consists of translates of the $g$ coordinate codimension $1$ tori in $\Jac (\Gamma)$, and
each of these carries a positive integer multiplicity.  By Proposition \ref{Prop:IntersectionNumber}, the intersection of $g$ general translates of
$\Theta_{\Gamma}$ consists of precisely $g!$ distinct points, and each of these points is endowed with a tropical intersection multiplicity $m$ equal to
the product of the multiplicities of the facets.  By Rabinoff's lifting theorem~\cite{Rabinoff12}, applied on the universal cover of $\Jac(\Gamma)$, $m
\cdot g!$ is the number of intersection points, counted with multiplicity, in the intersection of $g$ general translates of $\Theta_C \subset \Jac(C)$.
Because the theta divisor provides a principal polarization of the Jacobian, the intersection number $\Theta^g$ is $g!$.  It follows that $m$ is $1$, and
hence every facet of $\Theta_{\Gamma}$ has multiplicity $1$.
\end{proof}

\subsection{Local equations}

We begin by proving
an analogue of the local equations from
Proposition~\ref{Prop:TropicalLocalEqns}, but for algebraic curves. In
particular, if $D$ and~$E$ are effective divisors on an algebraic curve~$C$ then $r(D) - \deg (E) \leq r(D-E) \leq r(D)$, which gives a set-theoretic
containment of the corresponding
Brill-Noether loci. Our next result shows that this containment is in fact
scheme-theoretic.

\begin{proposition}\label{Prop:Containments}
Let $C$ be a curve.  Fix integers $d$ and $r$ and an effective divisor $E$ of
degree $e$.  Let $\varphi \colon \Pic_{d-e}(C) \to \Pic_d(C)$ be the isomorphism defined by sending $[D]$ to $[D+E]$.  Then there is a chain of inclusion of subschemes
$$ W^{r+e}_d (C) \subseteq \varphi ( W^r_{d-e} (C)) \subseteq W^r_d (C) . $$
\end{proposition}

\begin{proof}
We use the following scheme-theoretic description of Brill-Noether loci, the details of which can be found in Chapter IV.3 of \cite{ACGH}.  Choose
an auxiliary effective divisor $F$ of degree $m \geq 2g-d+e+1$, and let $\mathcal{L}$ be a Poincar{\'e} line bundle on $\Pic_d (C) \times C$.  We then have
an exact sequence of sheaves
$$ 0 \to \mathcal{L} \to \mathcal{L}(F) \to \mathcal{L}(F) \vert_F \to 0. $$
If $\nu \colon \Pic_d(C) \times C \to \Pic_d(C)$ is the projection map, then $W^r_d(C)$ is the determinantal variety defined by the $(m+d-g+1-r) \times
(m+d-g+1-r)$ minors of the map of vector bundles $\nu_* \mathcal{L}(F) \to \nu_* \mathcal{L}(F) \vert_F$.

Since $\varphi$ is an isomorphism, and $( \varphi \times \mathrm{id})^* \mathcal{L}(-E)$ is a Poincar{\'e} line bundle on $\Pic_{d-e}(C) \times C$, the ideal of
$\varphi (W^r_{d-e} (C))$ is defined by the $(m+d-e-g+1-r) \times (m+d-e-g+1-r)$ minors of $\nu_* \mathcal{L}(F-E) \to \nu_* \mathcal{L}(F-E) \vert_F$.
Note that $\mathcal{L}(F-E) \vert_F = \mathcal{L}(F) \vert_F$ and also that there is an injection $\mathcal{L}(F-E) \to \mathcal{L}(F)$, which pushes
forward to an injection.  Putting these together, we have an exact sequence
\begin{equation}\label{Eq:Pushforward}
0 \to \nu_* \mathcal{L}(F-E) \to \nu_* \mathcal{L}(F) \to \nu_* \mathcal{L}(F) \vert_F ,
\end{equation}
where $W^r_d (C)$ is defined by the minors of the last map, and $\varphi (W^r_{d-e} (C))$ is defined by the minors of the composition.

To finish the proof, we may work locally, so we replace~(\ref{Eq:Pushforward}) with an exact sequence of vector spaces:
$$ 0 \to U \to U \oplus V \to W . $$
Note that $U$ has dimension $d+m-g-e+1$ and $V$ has dimension $e$.  Consider the matrix representing the map $U \oplus V \to W$.  The ideal of its $(k+e)
\times (k+e)$ minors is contained in the ideal of $k \times k$ minors of the map $U \to W$, which is contained in the ideal of $k \times k$ minors of the
map $U \oplus V \to W$.  Taking $k = m+d-g-e+1-r$, gives the desired inclusion of schemes.
\end{proof}

\begin{corollary}
\label{Cor:LocalEqns}
Fix a curve $C$ and integers $d,r \geq 0$.  Let $E$ and $E'$ be effective divisors of degree $r$ and $g-1-d+r$, respectively.  Let $\varphi$
be the map from $\Pic_d (C)$ to $\Pic_{g-1} (C)$ taking $[D]$ to $[D-E+E']$.  Then $\varphi (W^r_d (C)) \subseteq \Theta_C$.
\end{corollary}

\subsection{Proof of main theorem}

We now prove Theorem~\ref{Thm:MainThm}, using
the lifting theorem from~\cite{Rabinoff12} for $0$-dimensional complete intersections.  When $\rho(g,r,d)$ is positive, we intersect $W^r_d(\Gamma)$ with
$\rho(g,r,d)$ translates of the theta divisor and use the counting formula from
Proposition~\ref{Prop:IntersectionNumber}.

\begin{proof}[Proof of Theorem~\ref{Thm:MainThm}]
We first consider the case where $[D] \in W^r_d(\Gamma)$ is a vertex avoiding divisor class and $K$ is algebraically closed.
By Proposition \ref{Prop:IntersectionNumber}, there exist divisors $E_{1},
\ldots, E_{\rho}$ of degree $d-g+1-r$ such that $[D]$ is contained in the
intersection
$$ X = \bigcap_{i=1}^{\rho} [\Theta_{\Gamma} + E_i ] \cap W^r_d ( \Gamma ) $$
and
$$ \vert X \vert = g! \prod_{i=0}^r \frac{i!}{(g-d+r+i)!} . $$
We let $\mathcal E_i$ be a divisor on~$C$ which tropicalizes to
$E_i$. We define $\mathcal X \subset \Pic_d(C)$ to be the intersection
$$ \mathcal{X} = \bigcap_{i=1}^{\rho} [\Theta_C + \mathcal{E}_i ] \cap W^r_d (C) , $$
which is a finite scheme because $\Trop(\mathcal X)$ is contained in $X$, which
is finite.

Now let $[D']$ be one of the finitely many points in~$X$, and we apply
Proposition~\ref{Prop:TropicalLocalEqns} to see that in a neighborhood of~$[D']$,
we can write $W^r_d(\Gamma)$ as an intersection
\begin{equation*}
\bigcap_{i=1}^{g-\rho} \Theta_{\Gamma} + [E_i' - E_i''],
\end{equation*}
where $E_i'$ and $E_i''$ are effective
divisors of degree $r$ and $d-g+1$ respectively. As before, we lift each of
these to effective divisors $\mathcal E_i'$ and $\mathcal E_i''$ on $C$.  By Corollary~\ref{Cor:LocalEqns}, $W^r_d(C)$ is contained in $\Theta_C +
[\mathcal E_i' - \mathcal E_i'']$ for each~$i$.

Near $[D']$, the hypersurfaces $\Trop(\Theta_C + [\mathcal E_i])$ and $\Trop(\Theta_C + [\mathcal E_i' - \mathcal E_i''])$
are translates of coordinate planes, all with
multiplicity~$1$ by Lemma~\ref{Lemma:MultiplicityOfTheta}. Therefore, it follows
from~\cite{Rabinoff12} that there is exactly one reduced point
tropicalizing to~$[D']$ in the intersection of
\begin{equation*}
\Theta_C + [\mathcal E_1], \ldots, \Theta_C + [\mathcal E_\rho], \mbox{ and }
\Theta_C + [\mathcal E_1' - \mathcal E_1''], \ldots, \theta_C + [\mathcal E_{g-\rho}' - \mathcal E_{g - \rho}''].
\end{equation*}
While these hypersurfaces contain $\mathcal X$, it's not true that they form a complete set of defining equations, so this only shows that tropicalization
is an injection from $\mathcal X$ to~$X$.
However, these two sets have the same cardinality, by
Proposition~\ref{Prop:IntersectionNumber} and \cite{ACGH}, so we
have a bijection. In particular, there exists a divisor class $[\mathcal D] \in
W^r_d(C)$ such that $\Trop([\mathcal D]) = [D]$.

Since $K$ is algebraically closed, the set of vertex avoiding divisors which are rational over the value group of~$K$ form a dense set of~$W^r_d(\Gamma)$, and the lifting argument above shows that each of these is in $\Trop(W^r_d(C))$.  Since $\Trop(W^r_d(C))$ is closed, it follows that $\Trop(W^r_d(C))$ is equal to $W^r_d(\Gamma)$.  Therefore, even if we drop the assumption that $[D]$ is vertex-avoiding, the preimage of any point in $W^r_d(\Gamma)$ that is rational over the value group of~$K$ is a nonempty strictly $K$-analytic domain in $W^r_d(C)^{\an}$.  Since $K$ is algebraically closed, the $K$-points are dense in this domain.

We now additionally drop our assumption that $K$ is algebraically closed.  The argument above still shows that the preimage of~$[D]$ is a nonempty strictly $K$-analytic domain $\mathcal W_{[D]}$ in $W^r_d(C)$, and it just remains to show that this domain contains a $K$-rational point.  Since $C$ has totally split reduction, the preimage of $[D]$ in $\Pic_d(C)^{\an}$ is an annulus, and is the generic fiber of a formal scheme whose special fiber is a torus $\mathbb G_m^g$ over the residue field.  Let $\mathfrak W$ be the closure of $\mathcal{W}_{[D]}$ in this formal scheme, and let $\mathfrak W_0$ be the special fiber of~$\mathfrak W$.  By construction, $\mathfrak W_0$ is a subscheme of $\mathbb G_m^g$ over the residue field, and by~\cite[Lem.~6.9.5]{Chambert-LoirDucros12}, the tropicalization $\Trop(\mathfrak W_0)$ is the local cone of $\Trop(W^r_d(C)) = W^r_d(\Gamma)$ at $[D]$.  The results of Section~\ref{sec:classification} show that this local cone is a union of $\rho$-dimensional coordinate linear subspaces, and the lifting argument above shows that each of these has multiplicity 1.

Choose one of these linear spaces and then the coordinates of the linear space define a map~$\pi$ from $\mathbb G_m^g$ to $\mathbb G_m^\rho$. Since
the facets of $\Trop(\mathfrak W_0)$ have multiplicity 1, this map has tropical degree 1, and hence $\pi$ is birational on one of the components of~$\mathfrak W_0$. Since the residue field is infinite, the rational points in $\mathbb G_m^\rho$ are Zariski dense.  In particular, there is a rational point in the dense open subset over which $\pi$ is an isomorphism.  Therefore, there is a smooth rational point in $\mathfrak W_0$, and this lifts to a $K$-point in $\mathfrak W$ by Hensel's Lemma.  In particular, there is a $K$-rational point of $W^r_d(C)$ in the preimage of $[D]$, as required.
\end{proof}

\begin{proof}[Proof of Prop.~\ref{Prop:Reduced}]
From the proof of Theorem~\ref{Thm:MainThm}, we know that the intersection of $W^r_d(C)$ with $\rho$ translates of the theta divisor consists of distinct
reduced points. Thus, $W^r_d(C)$ is smooth at those points and so it is generically reduced. However, $W^r_d(C)$ is a determinantal locus and hence
Cohen-Macaulay, so it has no embedded points, and thus $W^r_d(C)$ is everywhere reduced.
\end{proof}

\bibliography{math}

\end{document}